\documentclass[%
  a4paper,
  twocolumn,
  colorlinks,
]{preprint}

\usepackage{graphicx}
\usepackage{tikz}
\usepackage{pgfplots}
\pgfplotsset{compat = 1.17}
\pgfkeys{/pgf/number format/.cd,1000 sep={\,}}
\tikzstyle{sline} = [
  solid,
  line width = 1.5pt
]

\pgfplotscreateplotcyclelist{plotlistmsd}{
  {blue, sline},
  {orange, sline, dashed},
  {purple, sline, dashed},
  {green, sline, dotted},
  {red, sline, dotted},
}

\pgfplotscreateplotcyclelist{plotlistmsderror}{
  {orange, sline, dashed},
  {purple, sline, dashed},
  {green, sline, dotted},
  {red, sline, dotted},
}

\pgfplotscreateplotcyclelist{plotlistconvdiff}{
  {blue, sline},
  {orange, sline, dashed},
}

\pgfplotscreateplotcyclelist{plotlistconvdifferror}{
  {red, sline},
}

\newcommand{\plotfontsize}{\footnotesize}

\usepackage{appendix}
\usepackage{amsthm}
\usepackage{bbm}
\usepackage{epsfig}
\usepackage{color}
\usepackage{amsmath,amssymb}
\usepackage{comment}
\usepackage{mathdots}
\usepackage{mathtools}
\usepackage{wrapfig}
\usepackage{enumitem}
\usepackage{algorithm}
\usepackage{algpseudocode}

\algrenewcommand\algorithmicrequire{\textbf{Input:}}
\algrenewcommand\algorithmicensure{\textbf{Output:}}

\newcommand{\bx}{\textbf{x}}
\newcommand{\bA}{\textbf{A}}
\newcommand{\bR}{\textbf{R}}
\newcommand{\bQ}{\textbf{Q}}
\newcommand{\bb}{\textbf{b}}
\newcommand{\bc}{\textbf{c}}
\newcommand{\bC}{\textbf{C}}
\newcommand{\bB}{\textbf{B}}

\newcommand{\bI}{\textbf{I}}

\newcommand{\bU}{\textbf{U}}
\newcommand{\bW}{\textbf{W}}
\newcommand{\bV}{\textbf{V}}
\newcommand{\bw}{\textbf{w}}
\newcommand{\bT}{\textbf{T}}

\newcommand{\bf}{\textbf{f}}

\newcommand{\bu}{\textbf{u}}
\newcommand{\bX}{\textbf{X}}
\newcommand{\bD}{\textbf{D}}
\newcommand{\by}{\textbf{y}}

\newcommand{\bs}{\textbf{s}}
\newcommand{\bv}{\textbf{v}}

\newcommand{\be}{\textbf{e}}

\renewcommand{\i}{i\mkern1mu}
\renewcommand{\vec}{\operatorname{vec}}

\newcommand{\hbx}{\hat{\textbf{x}}}
\newcommand{\hbA}{\hat{\textbf{A}}}
\newcommand{\hbB}{\hat{\textbf{B}}}
\newcommand{\hbc}{\hat{\textbf{c}}}

\newcommand{\hy}{\hat{y}}

\newcommand{\R}{\mathbb{R}}
\newcommand{\C}{\mathbb{C}}

\newcommand{\cH}{\mathcal{H}}
\newcommand{\cP}{\mathcal{P}}
\newcommand{\cL}{\mathcal{L}}
\newcommand{\cO}{\mathcal{O}}
\newcommand{\cE}{\mathcal{E}}

\newcommand{\cS}{\mathcal{S}}
\newcommand{\cW}{\mathcal{W}}
\newcommand{\cQ}{\mathcal{Q}}
\newcommand{\cB}{\mathcal{B}}

\newcommand{\ccE}{\mathcal{E}_c}
\newcommand{\coE}{\mathcal{E}_o}
\newcommand{\dt}{\operatorname{d}\hspace{-0.05cm}t}
\newcommand{\pd}[2]{\frac{\partial #1}{\partial #2}}
\newcommand{\Sym}{\operatorname{Sym}}

\newcommand{\krank}{\operatorname{krank}}

\newcommand{\tbW}{\widetilde{\bW}}
\newcommand{\tbw}{\widetilde{\bw}}

\newcommand{\tbC}{\widetilde{\bC}}

\newcommand{\tbQ}{\widetilde{\bQ}}
\newcommand{\tbx}{\widetilde{\bx}}
\newcommand{\tcW}{\widetilde{\cW}}
\newcommand{\tcoE}{\widetilde{\coE}}
\newcommand{\tccE}{\widetilde{\ccE}}

\DeclareMathOperator*{\argmax}{arg\,max}

\newtheorem{lemma}{Lemma}
\newtheorem{theorem}{Theorem}
\newtheorem{proposition}{Proposition}

\newtheorem{rmk}{Remark}

\begin{document}

\title{Energy-Based Approximation of Linear Systems with Polynomial Outputs}

\author[$\ast$]{Linus Balicki}
\author[$\ast$]{Serkan Gugercin}
\affil[$\ast$]{Department of Mathematics, Virginia Tech, Blacksburg, VA, 24061, USA}
          

\abstract{
    Controllability and observability energy functions play a fundamental role in model order reduction  and are inherently connected to optimal control problems. For linear dynamical systems the energy functions are known to be quadratic polynomials and various low-rank approximation techniques allow for computing them in a large-scale setting. For nonlinear problems computing the energy functions is significantly more challenging. In this paper, we investigate a special class of nonlinear systems that have a linear state and a polynomial output equation. We show that the energy functions of these systems are again polynomials and investigate under which conditions they can effectively be approximated using low-rank tensors. Further, we introduce a new perspective on the well-established balanced truncation method for linear systems which then readily generalizes to the nonlinear systems under consideration. This new perspective yields a novel energy-based model order reduction procedure that accurately captures the input-output behavior of linear systems with polynomial outputs via a low-dimensional reduced order model. We demonstrate the effectiveness of our approach via two numerical experiments.
}

\novelty{}

\maketitle

\section{Introduction}
In this work we investigate dynamical systems 
that are {linear in the dynamics and polynomial in the output}, that is, dynamical systems of the form
\begin{equation}
\label{eq:posys}
\begin{aligned}
    \dot{\bx}(t) &=\bA \bx(t) + \bB \bu(t) \\
    y(t) &= \bc_1^\top \bx(t) \\
    &\, + \bc_2^\top \left(\bx(t) \otimes \bx(t) \right) \\
    &\, + \cdots \\
    &\, + \bc_d^\top \left(\bx(t) \otimes \cdots \otimes \bx(t) \right),
\end{aligned}
\end{equation}
where $\bA \in \mathbb{R}^{n \times n}$, $\bB \in \mathbb{R}^{n \times m}$, and $\bc_j \in \mathbb{R}^{n^j}$ for $j= 1,\ldots,d$. Here, $\bx(t) \in \mathbb{R}^n$ is the state, $\bu(t) \in \mathbb{R}^m$ the input and $y(t) \in \mathbb{R}$ the output of the system. We note that our discussion can easily be extended to systems with multiple outputs (i.e., $\by(t)\in\R^p$) as we later discuss in Remark~\ref{rmk:MIMO}. However, in order to allow for a more concise presentation of our results, we will only consider the single-output case for now. We will assume that the system~\eqref{eq:posys} is asymptotically stable, i.e., the eigenvalues of $\bA$ lie in the open left-half plane, i.e.,
$\Lambda(\bA) \in \mathbb{C}_{-}$.\\

While the state equation of~\eqref{eq:posys} corresponds to that of a standard linear time-invariant (LTI) system, the output is a degree-$d$ polynomial in the state. Dynamical systems that combine the output of an LTI system with a static nonlinear function are commonly referred to as Wiener models. The system in \eqref{eq:posys} is therefore a special type of Wiener model where the static nonlinearity corresponds to a polynomial. To make a clear distinction between Wiener models with arbitrary nonlinear outputs and systems with polynomial outputs we will refer to the model in \eqref{eq:posys} as an LTI polynomial output (LPO) system. Systems of this type are widely used in nonlinear system identification \cite{Ogunfunmi07,Kalafatis95,Gomez04}. {Furthermore, polynomial outputs allow for locally approximating general nonlinear outputs arbitrarily well \cite{Boyd85}.} Also, in the context of port-Hamiltonian systems \cite{Mehrmann23} it can be of interest to investigate how a Hamiltonian associated with a system evolves with respect to various inputs. In many cases the Hamiltonian can be modelled as a quadratic polynomial leading to an LPO system~\cite{Mehrmann23}. Quadratic outputs also appear in the context of modelling the variance of a quantity of interest of a stochastic model \cite{Haasdonk13}. Systems with quadratic outputs have recently been studied in the context of model order reduction (MOR); see, e.g., \cite{Reiter2024,Gosea2022,Przybilla2024,Pulch2019,Beeumen2012,Gosea19,Haasdonk13,Benner22}.

The controllability energy function 
\begin{equation}
\label{eq:controllabilityenergy}
    \ccE(\bx_0) = \min_{\substack{u \in L_2(-\infty,0] \\ \bx(-\infty)=0 \\ \bx(0)=\bx_0}} \; \; \frac{1}{2} \int_{-\infty}^{0} \lVert \bu(t) \rVert^2 \dt
\end{equation}
and the observability energy function
\begin{equation}
    \label{eq:observabilityenergy}
   \coE(\bx_0) = \frac{1}{2} \int_{0}^{\infty} \lVert y(t) \rVert^2 \dt, \, \bx(0) = \bx_0, \, u(t) \equiv 0
\end{equation}
play a fundamental role in various MOR methods. Specifically, these energy functions are integral parts of the linear and nonlinear balanced truncation (BT) method \cite{mullis1976synthesis,Moore81,Scherpen93,Fujimoto10}. The energy functions are also inherently connected to optimal control problems. In the context of optimal control theory we typically aim to find an optimal control law $\bu^\star$ which minimizes a cost functional such as
\begin{equation*}
    \cE(\bx,\bu) = \int_{0}^\infty h(\bx(t)) + \gamma \lVert \bu(t) \rVert^2 \dt.
\end{equation*}
The energy functions can therefore be considered value functions of optimal control problems, see, e.g., ~\cite{Lewis2012,Breiten2019,BorgZ2020}. While BT is successfully applied to large-scale linear systems, developing computational frameworks for nonlinear systems is still an active  area of research, see, e.g., \cite{Kawano2017,Fujimoto10,Kramer23,Sarkar2024,Corbin2024,Corbin2024_2} and the references therein.
A key bottleneck in the nonlinear balancing framework \cite{Fujimoto10,Kramer23} is the computation of the controllability and observability energy functions. Especially for large $n$ storage and computation of the functions become a critical issue. Applying MOR techniques is particularly useful when the dimension of the model of interest is large ($n \gg 1000$). In this case reduced order models (ROMs) approximate the input-output mapping based on a much lower-dimensional system enabling fast simulations and design of model-based controllers. However, solving optimal control problems and hence computing energy functions for systems with many variables is known to be notoriously difficult and the expression ``curse of dimensionality'' is often used to emphasize this fact. For general nonlinear systems various tensor-based techniques exist to solve optimal control problems \cite{Almubarak2019,Breiten2019,Borggaard2021,Dolgov21,Dolgov23}. Although these generic solution approaches are flexible regarding the structure of the nonlinearity, they are typically limited to systems where the state-space dimension does not exceed a few hundred.  The first central goal of this manuscript is to establish conditions under which the controllability and observability energy functions for LPO systems can be computed and stored efficiently in a large-scale setting. A crucial step in achieving this goal is to exploit the structure of LPO systems in our computational procedure. Beside the computational effort associated with computing the energy functions, another challenge in the nonlinear BT method is that it does not preserve the structure of an underlying nonlinearity in the reduction process. In particular, if nonlinear BT is applied to an LPO system as in~\eqref{eq:posys}, the ROM will in general not have a linear state equation anymore. This makes interpretation and analysis of the ROM difficult and raises practical questions regarding efficient evaluation of the ROM. Motivated by this fact, we introduce a new energy-based MOR procedure that preserves the structure of an LPO system in the reduction process. We show that this new procedure can be viewed as a generalization of the linear BT approach and demonstrate that it performs well in practice. \\

{This paper is organized as follows: In section~\ref{sec:energyfunctions} we derive our first main result that reveals that energy functions of LPO systems are polynomials. Additionally, explicit formulae are provided which allow for computing the polynomial coefficients of the observability energy functions via solving a set of linear systems of equations. Section~\ref{sec:tensoralgebra} briefly introduces tensor decompositions that will be important for our subsequent discussions. A computational procedure and conditions under which the polynomial coefficients of $\coE$ can be computed in a large-scale setting are discussed in section \ref{sec:lrsolutions}. Section~\ref{sec:energybasedMOR} revisits the BT method for linear systems and introduces a new optimization-based perspective for it. Another main contribution is the energy-based MOR algorithm for LPO systems introduced in section~\ref{sec:energybasedapproximation}. It is a natural extension of the optimization-based perspective on BT introduced in section~\ref{sec:energybasedMOR} and presents a novel way for utilizing energy functions in the context of nonlinear MOR. In section~\ref{sec:numericalexamples} we demonstrate the effectiveness of our proposed method via numercial examples. Finally, we summarize our contributions in section~\ref{sec:conclusion}.}

\section{Energy Functions}
\label{sec:energyfunctions}
It is well known that solving optimal control problems is directly related to solving Hamilton-Jacobi-Bellman-type equations~\cite{Lewis2012}. {In the special case of an LPO system as in \eqref{eq:posys}, the energy functions in \eqref{eq:controllabilityenergy} and \eqref{eq:observabilityenergy} are solutions to the Hamilton-Jacobi and Lyapunov PDEs} \cite{Scherpen93}
\begin{align}
    0 &= \pd{\ccE(\bx)}{\bx} \bA \bx + \frac{1}{2} \pd{\ccE(\bx)}{\bx} \bB \bB^\top \pd{\ccE(\bx)^\top}{\bx}~\mbox{and} \label{eq:controllabilityPDE}\\
    0 & = \pd{\coE(\bx)}{\bx} \bA \bx + \frac{1}{2} \sum_{j=1}^d \sum_{k=1}^d (\bc_j \otimes \bc_k )^\top (\bx \otimes \cdots \otimes \bx). \label{eq:observabilityPDE}
\end{align}
Here, the controllability energy corresponds to that of a standard LTI system and, with moderate assumptions,~\eqref{eq:controllabilityPDE} has a unique solution, which is a quadratic polynomial in the form
\begin{equation*}
    \ccE(\bx) = \frac{1}{2} \bx^\top \cP^{-1} \bx.
\end{equation*}
The symmetric positive-definite matrix $\cP \in \mathbb{R}^{n \times n}$ is the controllability gramian and solves the Lyapunov matrix equation
\begin{equation}
\label{eq:controllyap}
    	0 = \bA \cP + \cP \bA^\top + \bB \bB^\top.
\end{equation}
The nonlinear output of the system in \eqref{eq:posys} only plays a role for the PDE in \eqref{eq:observabilityPDE}. Since approximation techniques for the linear matrix equation in \eqref{eq:controllyap} are rather well understood \cite{Simoncini16,Benner2013}, we now focus solely on the observability energy. Our first main result, which gives an explicit representation for \eqref{eq:observabilityenergy}, is given in the next theorem.
\begin{theorem}
    \label{theorem:observabilityenergy}
    Let $\Lambda(\bA) \subset \C_-$. Then the observability energy function of the system in \eqref{eq:posys} is  the unique solution to the Lyapunov PDE in \eqref{eq:observabilityPDE} and satisfies the formula
    \begin{equation}
    \label{eq:observabilitypolynomial}
        \coE(\bx) = \frac{1}{2} \sum_{j=2}^{2d} \bw_j^\top (\bx \otimes \cdots \otimes \bx);
    \end{equation}
    and thus is a degree-$2d$ polynomial. A set of polynomial coefficients $\left\{ \bw_j \right\}_{j=2}^{2d}$, which defines $\coE$, is given by the unique solutions to the linear systems
    \begin{equation}
    \label{eq:kroneckersystem}
    \begin{aligned}
        \cL_k(\bA^\top) \bw_k = - \sum_{i = 1}^{k-1} \bc_i \otimes \bc_{k-i},
    \end{aligned}
    \end{equation}
    where $\cL_k(\bA^\top) \in \mathbb{R}^{n^k \times n^k}$ is defined as
    \begin{equation*}
    \cL_k(\bA^\top) = \sum_{j=1}^k \bI \otimes \cdots \otimes \bI \otimes \underbrace{\bA^\top}_{j\text{-th term}} \otimes \bI \otimes \cdots \bI.
\end{equation*}
\end{theorem}
\begin{proof}
    Consider the observability energy defined in \eqref{eq:observabilityenergy} and note that $\bu(t) \equiv 0$ and $\bx(0) = \bx_0$. Due to the assumption $\Lambda(\bA) \subset \C_-$, the state of the LPO system in \eqref{eq:posys} can be written explicitly as $\bx(t) = e^{\bA t} \bx_0$. Plugging the resulting expression into $y(t)$ gives
    \begin{equation*}
    \begin{aligned}
                y(t) &= \bc_1^\top e^{\bA t} \bx_0 + \ldots + \bc_d^\top \left(e^{\bA t} \bx_0 \otimes \cdots \otimes e^{\bA t} \bx_0 \right).
    \end{aligned}
    \end{equation*}
    This allows for writing the output energy as
    \begin{equation*}
        \lVert y(t) \rVert^2 = \sum_{k=2}^{2d} \sum_{i = 1}^{k-1} (\bc_i \otimes \bc_{k-i})^\top \\
        \overbrace{(e^{\bA t} \bx_0 \otimes \cdots \otimes e^{\bA t} \bx_0)}^{k \text{ terms}}
    \end{equation*}
    and thus yields the expression 
    \begin{equation*}
        \begin{aligned}
        \coE(\bx_0) = \frac{1}{2} \sum_{k=2}^{2d} \sum_{i = 1}^{k-1} & (\bc_i \otimes \bc_{k-i})^\top\\
        \int_0^\infty (e^{\bA t}  \otimes & \cdots \otimes e^{\bA t} )\dt \left( \bx_0 \otimes \cdots \otimes \bx_0 \right). \\
        \end{aligned}
    \end{equation*}
    The representation above immediately shows that $\coE$ is a degree-$2d$ polynomial in $\bx_0$. Further, the degree-$k$ polynomial coefficient of $\coE$ is given by 
    \begin{equation*}
        \bw_k^\top = \sum_{i = 1}^{k-1} (\bc_i \otimes \bc_{k-i})^\top \int_0^\infty (e^{\bA t}  \otimes \cdots \otimes e^{\bA t} )\dt.
    \end{equation*}
    Next, we use the fact that $\Lambda(\bA) \subset \C_-$ implies
    \begin{equation}
    \label{eq:inverseintegral}
        \cL_k(\bA)^{-1} = -\int_0^\infty (e^{\bA t}  \otimes \cdots \otimes e^{\bA t} )\dt,
    \end{equation}
    see, e.g., \cite{Grasedyck03}. We then obtain
    \begin{equation*}
        \bw_k^\top = -\sum_{i = 1}^{k-1} (\bc_i \otimes \bc_{k-i})^\top \cL_k(\bA)^{-1}.
    \end{equation*}
    Taking the transpose and multiplying both sides of the above equation with $\cL_k(\bA^\top)$ concludes the proof.
\end{proof}
{The preceding theorem offers two significant insights into the energy functions of LPO systems. First, the energy functions are exactly represented by polynomials of degree $2d$. This is in stark contrast to the general nonlinear setting, where the observability energy functions are  analytic function and are typically merely approximated by polynomials \cite{Kramer24,Dolgov21}. Second, Theorem \ref{theorem:observabilityenergy} shows that we can compute observability energy functions explicitly by solving  linear systems of equations. Importantly, these linear systems are uniquely solvable implied by the stability condition $\Lambda(\bA) \subset \C_-$ \cite{Kramer24,Grasedyck03}.} 
A key tool in the proof of the theorem is the relationship in \eqref{eq:inverseintegral}. We will revisit this equation in Section~\ref{sec:lrsolutions} as it will constitute the basis of our computational procedure. For moderately sized problems (e.g., $n \leq 1000$ and $d \leq 2$) explicit computation of the polynomial coefficients $\bw_k \in \mathbb{R}^{n^k}$ is typically tractable. For large-scale systems however, even storing the polynomial coefficients is an infeasible task. A widely used approach to overcome this obstacle is to store the polynomial coefficients implicitly in terms of a low-rank tensor decomposition. Note that the polynomial coefficients $\bw_k$ correspond to a $k$-th order tensor $\cW_k \in \mathbb{R}^{n \times \cdots \times n}$, which can be defined via the relationship $\vec(\cW_k) = \bw_k$ \cite{Kolda09}. This connection motivates studying efficient storage formats for tensors in order to enable scalable computations for the polynomial coefficients of $\coE$. As we will discuss in detail later, the coefficients that define a multivariate polynomial as in  \eqref{eq:observabilitypolynomial} are not unique. We will show that there exist other uniquely solvable linear systems, which yield sets of polynomial coefficients that also define $\coE$. Another aspect of our discussion will reveal which of the linear systems should be considered when effective storage via low-rank tensor decompositions is a central goal.
\begin{rmk}
\label{rmk:MIMO}
{In many applications we are interested in systems with multiple inputs and multiple outputs (MIMO systems). In particular, consider the system in \eqref{eq:posys} with $p$ polynomial outputs given by $\by(t) = [y_1(t),\ldots,y_p(t)]^\top \in \R^p$, where $y_j(t) = \sum_{i=1}^{d_j} \bc_{ji}^\top \left( \bx(t) \otimes \cdots \otimes \bx(t) \right)$. In this case we can follow the proof of Theorem~\ref{theorem:observabilityenergy} to show that the observability energy function of LPO MIMO systems is again a polynomial along the lines of \eqref{eq:observabilityenergy} where the coefficients now solve linear systems of the form
    \begin{equation*}
        \cL_k(\bA^\top) \bw_k = -\sum_{j=1}^p \sum_{i=1}^{d_j} \bc_{ji} \otimes \bc_{j(k-i)}.
    \end{equation*}
    To make our discussion easier to follow we will stick to the single output case and note that all of our results can easily be extended to the MIMO setting.}
\end{rmk}

\section{Tensor Algebra Preliminaries}
\label{sec:tensoralgebra}
In the following we consider a $k$-th order tensor
\begin{equation*}
    \cW \in \mathbb{R}^{\overbrace{n \times \cdots \times n}^{k \text{ terms}}}
\end{equation*}
with uniform mode size $n$. We begin by reviewing the canonical decomposition of such a tensor, which plays an important role in our computational procedure.
\subsection{Canonical Decomposition}
A canonical (CP) decomposition \cite{Hitchcock27,Kolda09} for $\cW$ is defined via matrices $\bU_1,\ldots,\bU_k \in \mathbb{R}^{n \times R}$ which allow for the element-wise representation 
\begin{equation*}
    \cW(i_1,\ldots,i_k) = \sum_{j=1}^R \bU_1(i_1,j) \cdots \bU_k(i_k,j).
\end{equation*}
The matrices $\bU_1,\ldots,\bU_k$ are called the CP factors and the smallest $R$ for which such a decomposition exists defines the tensor rank of $\cW$. Note that the tensor $\cW$ has $n^k$ entries while the CP factors are formed by a total of $nkR$ entries. Hence, if a CP decomposition for $\cW$ with small $R$ exists we can use it to store the tensor efficiently even for large $n$ and $k$. A representation of the CP decomposition which will be useful in our analysis is that of a vectorized tensor given by the sum of Kronecker products
\begin{equation}
    \label{eq:veccp}
    \bw := \vec(\cW) = \sum_{j=1}^R \bU_1(:,j) \otimes \cdots \otimes \bU_k(:,j) \in \mathbb{R}^{n^k},
\end{equation}
{where $\bU_i(:,j) \in \R^n$ denotes the $j$-th column vector of the matrix $\bU_i$.} We will carry out the majority of our theoretical discussion based on the above Kronecker structured representation of a tensor. Due to its direct connection to the tensor CP decomposition we call the representation in \eqref{eq:veccp} the CP decomposition for $\bw$. Following the definition in \cite{Grasedyck03} we define the Kronecker rank of a vector $\bw \in \R^{n^k}$ as the smallest $R$ such that $\bw$ exhibits a (vectorized) CP decomposition along the lines of \eqref{eq:veccp}. In the following we denote the Kronecker rank of a vector $\bw$ as $\krank(\bw)$. We point out that the Kronecker rank of $\bw$ and the tensor rank of $\cW$ are equal.
\subsection{Symmetric Tensors}
\label{sec:symmetrictensors}
For the nonlinear BT framework \cite{Kramer23} as well as our proposed MOR approach, symmetric tensors \cite{Comon08} will play an important role. We begin by defining the symmetrization of a vector $\bu_1 \otimes \cdots \otimes \bu_k$ with Kronecker rank $1$ as
\begin{equation}
    \label{eq:symmetrization}
    \Sym\left[ \bu_1 \otimes \cdots \otimes \bu_k \right] = \frac{1}{k!}\sum_{\tau \in \Pi_k} \bu_{\tau(1)} \otimes \cdots \otimes \bu_{\tau(k)},
\end{equation}
where $\Pi_k$ denotes the symmetric group of permutations of $\{ 1,\ldots,k \}$.
The symmetrization of a general vector in $\R^{n^k}$ is then defined as the symmetrization of the terms that constitute its CP decomposition. For example, we have
\begin{equation*}
    \Sym\left[\bw\right] = \frac{1}{k!} \sum_{j=1}^R \sum_{\tau \in \Pi_k} \bU_{\tau(1)}(:,j) \otimes \cdots \otimes \bU_{\tau(k)}(:,j).
\end{equation*}
An important observation is that the sum above now consists of $k!R$ terms. Hence, symmetrization can increase the Kronecker rank by a factor of up to $k!$. We call $\bw$ (and the corresponding tensor $\cW$) symmetric if it satisfies $\Sym\left[ \bw \right] = \bw$. Equivalently, $\bw$ is symmetric if any permutation of its CP factors yields the same vector:
\begin{align*}
    \bw &= \sum_{j=1}^R \bU_1(:,j) \otimes \cdots \otimes \bU_k(:,j) \\
    &= \sum_{j=1}^R \bU_{\tau(1)}(:,j) \otimes \cdots \otimes \bU_{\tau(k)}(:,j) \; \text{ for all } \; \tau \in \Pi_k.
\end{align*}
In order to ease notation we define $\tau(\bw)$ as the vector where $\tau \in \Pi_k$ is applied to the ordering of the $k$ CP factors:
\begin{equation*}
    \tau(\bw) := \sum_{j=1}^R \bU_{\tau(1)}(:,j) \otimes \cdots \otimes \bU_{\tau(k)}(:,j).
\end{equation*}
Another useful fact about symmetric tensors is that they exhibit a symmetric CP decomposition 
\begin{equation*}
    \bw = \sum_{j=1}^R \bU(:,j) \otimes \cdots \otimes \bU(:,j)
\end{equation*}
in which all CP factors are given by the same matrix $\bU$. Note that the symmetrization of an arbitrary tensor will typically not yield such a form directly. While a symmetric CP decomposition exists for any symmetric tensor it is typically difficult to compute. Regardless, the symmetric CP decomposition will be a useful theoretical tool for us.
\section{Computing the Observability Energy Function}
\label{sec:lrsolutions}
Theorem \ref{theorem:observabilityenergy} shows that solving the Lyapunov PDE in \eqref{eq:observabilityPDE} amounts to solving the $2d-1$ linear systems in \eqref{eq:kroneckersystem}. In the following, we revisit the computational procedure introduced in \cite{Grasedyck03} which leverages the CP decomposition in order to approximately solve the linear systems in \eqref{eq:kroneckersystem}. Then we show how to take additional advantage of the structure of LPO systems in order to reduce the required storage of the polynomial coefficients $\bw_k$ even further. We emphasize that this discussion is essential for effectively carrying out the computations necessary for our MOR procedure we describe in Section~\ref{sec:energybasedapproximation}.
\subsection{Low-Rank Solutions to Linear Systems}
We begin by recalling a result from \cite{Grasedyck03,Kressner10}
which reveals under which conditions it is reasonable to expect approximations based on vectors with low Kronecker rank to be effective. Additionally, this result yields a formula for approximate solutions to linear systems as they appear in \eqref{eq:kroneckersystem}.
\begin{theorem}
    \label{theorem:quadratureapproximation}
    Let $\bA \in \mathbb{R}^{n \times n}$ be diagonalizable with the eigendecomposition $\bA = \bX \bD \bX^{-1}$ and $\Lambda(\bA) \in [-\lambda_{max},\allowbreak -\lambda_{min}] \oplus \i \left[ -\mu,\mu \right]$, where $\lambda_{min} > 0$. Let $\bb \in \mathbb{R}^{n^k}$ and $\cL_k(\bA) \bx = \bb$. Define $h_{st} = \frac{\pi}{\sqrt{\ell}}$ for $\ell \in \mathbb{N}$,
    \begin{align*}
        \alpha_i &= \log\left( \exp(i h_{st}) + \sqrt{1 + \exp(2 i h_{st}})\right) \Big/ (k \lambda_{min} ),~\mbox{and} \\
        \omega_i &= h_{st} \Big/ \left(\sqrt{1 + \exp(-2 i h_{st})} k \lambda_{min} \right),
    \end{align*}
    for $i = -\ell,\ldots,\ell$. Then the quadrature-based approximation 
    \begin{equation}
        \label{eq:quadrature}
        \bx_\ell = - \sum_{i =-\ell}^{\ell} \omega_i \left(\exp\left(\alpha_i \bA\right) \otimes \cdots \otimes \exp\left(\alpha_i \bA\right) \right) \bb
    \end{equation}
    to the solution of $\cL_k(\bA) \bx = \bb$
    satisfies the error bound
    \begin{equation}
        \label{eq:quadratureerror}
        \left\lVert \bx - \bx_\ell \right\rVert \leq \frac{\widetilde{C} \kappa_2(\bX)^k}{k \lambda_{min}} \exp\left( \frac{\mu}{ \lambda_{min} \pi} \right) \exp(-\pi \sqrt{\ell}) \lVert \bb \rVert,
    \end{equation}
    where $\widetilde{C} > 0$ is independent of $\ell$, $k$, $\bA$ and $\bb$. If $\bA$ is symmetric, the bound simplifies to
    \begin{equation*}
        \left\lVert \bx - \bx_\ell \right\rVert \leq \frac{\widetilde{C}}{k \lambda_{min}} \exp(-\pi \sqrt{\ell}) \lVert \bb \rVert.
    \end{equation*}
\end{theorem}
The proof of the theorem can be found in \cite{Kressner10} for a more general class of Kronecker structured linear systems. The key idea of the proof is to use a quadrature rule to approximate the integral in \eqref{eq:inverseintegral}. The corresponding quadrature nodes are denoted by $\alpha_j$ and quadrature weights by $\omega_j$. Note that if the right-hand side  $\bb$ is given in terms of a CP decomposition 
\begin{equation*}
    \bb = \sum_{j=1}^R \bb_1(:,j) \otimes \cdots \otimes \bb_k(:,j),
\end{equation*}
then the solution $\bx_\ell$ can be approximated efficiently and with arbitrary accuracy based on the formula in \eqref{eq:quadrature}:
\begin{alignat*}{2}
    \bx_\ell &= - \sum_{i =-\ell}^{\ell} \sum_{j=1}^R \omega_i \big(\exp \left(\alpha_i  \bA\right) && \bb_1(:,j) \otimes \cdots \\
    & && \otimes \exp\left(\alpha_i \bA\right)\bb_k(:,j) \big).
\end{alignat*}
Most importantly, the formula above represents $\bx_\ell$ directly in the CP format. The resulting Kronecker rank of the approximated solution is thus bounded by $(2\ell+1)R$ while the approximation error in \eqref{eq:quadratureerror} decays at a rate of $\cO(\exp(-\sqrt{\ell}))$. This means good low-rank tensor approximations are guaranteed to exist if the error bound in Theorem \ref{theorem:quadratureapproximation} is small for a relatively small $\ell$ and the right-hand side $\bb$ is a low-rank tensor. Whether such $\ell$ exist or not primarily depends on the spectral properties of the matrix $\bA$. For example, if $\bA$ is symmetric and none of its eigenvalues are close to the origin we can obtain small error bounds using small $\ell$. This is, for example, satisfied for the discretized Laplace operator, which is well-known to allow for good approximations in the context of Kronecker structured linear systems \cite{Hackbusch06}. Aside from the spectral properties of $\bA$, the Kronecker ranks of the right-hand side of the linear systems in \eqref{eq:kroneckersystem} significantly impact the Kronecker rank of the energy function coefficients $\bw_k$.
The next result examines the Kronecker rank of the right-hands of the linear system~\eqref{eq:kroneckersystem} arising in the computation of the observability energy function of LPO systems. 
\begin{proposition}
    Consider the polynomial coefficients $\bc_1,\allowbreak\ldots,\bc_d$ of the LPO system in \eqref{eq:posys}. Assume that the Kronecker ranks of $\bc_2,\ldots,\bc_d$ are less than or equal to $R$. Then the Kronecker rank of the right-hand side of the linear system~\eqref{eq:kroneckersystem} in Theorem \ref{theorem:observabilityenergy}, namely
    \begin{equation}
        \label{eq:rhs}
        \bC_k := - \sum_{i = 1}^{k-1} \bc_i \otimes \bc_{k-i},
    \end{equation}
    is bounded by $(k-3)R^2 + 2R$ for $k \geq 3$ and equal to $1$ for $k=2$.
\end{proposition}
The proposition follows from performing various arithmetic operations using Kronecker products. As a result we obtain a connection between the Kronecker ranks of the polynomial coefficients that constitute the output in \eqref{eq:posys} and the Kronecker ranks of the right-hand side of the linear systems in \eqref{eq:kroneckersystem}. In particular, we need to assume that the polynomial output coefficients are vectors with small Kronecker ranks in order for accurate low-rank approximations of the $\bw_k$ to exist. We note that the Kronecker rank of the right-hand side $\bb$ also plays a role in the computational complexity associated with evaluating the quadrature formula in \eqref{eq:quadrature}. In particular, it is necessary to evaluate the matrix exponential $\exp(\alpha_i\bA)$ applied to the $kR$ vectors $\bb_1(:,1),\ldots,\bb_k(:,R)$. If $n$ is not too large and $\bA$ is diagonalizable, we can compute, e.g., the eigendecomposition $\bA = \bX \bD \bX^{-1}$ and write
\begin{equation*}
    \exp(\alpha_i\bA) = \bX^{-\top} \exp(\alpha_i \bD) \bX^\top.
\end{equation*}
Given such a decomposition we can evaluate the quadrature rule by solving a linear system of equations with $\ell kR$ right-hand sides. Similar computations can be performed via the Schur decomposition of $\bA$ as well. If, on the other hand, the dimension of the problem does not allow for computing a full decomposition of $\bA$, one typically resorts to Krylov subspace methods to compute the application of the matrix exponential to a vector \cite{Marlis1997}. In either case, it is beneficial to work with vectors that have a low Kronecker rank in order to keep computations tractable.
\subsection{Exploiting Symmetries}
While the solutions to the linear systems in \eqref{eq:kroneckersystem} yield a set of polynomial coefficients $\{ \bw_j\}_{j=2}^{2d}$ that define $\coE$, there exist many other vectors that define the same polynomial. The following result characterizes other linear systems that can be solved for the different sets of polynomial coefficients. 
\begin{lemma}    \label{proposition:permutedsolutions}
    Let $\bA \in \mathbb{R}^{n \times n}$,  $\bb \in \mathbb{R}^{n^k}$ and $k$ such that the linear system
    \begin{equation*}
        \cL_k(\bA) \bw = \bb
    \end{equation*}
    has a unique solution $\bw \in \mathbb{R}^{n^k}$. Then for any symmetric permutation $\tau \in \Pi_k$, we have
    \begin{equation*}
        \cL_k(\bA) \tau(\bw) = \tau(\bb).
    \end{equation*}
    Further, $\bw$ and $\tau(\bw)$ define the same degree-$k$ homogenous polynomial
    \begin{equation*}
        \bw^\top\left( \bx \otimes \cdots \otimes \bx\right) = \tau(\bw)^\top\left( \bx \otimes \cdots \otimes \bx\right).
    \end{equation*}
\end{lemma}
\begin{proof}
    Consider the CP decomposition
    \begin{equation*}
        \bw = \sum_{j = 1}^R \bf_{1}^{(j)} \otimes \cdots \otimes \bf_{k}^{(j)}.
    \end{equation*}
    First, for some $j \in \left\{ 1,\ldots,R \right\}$ and $\tau \in \Pi_k$, consider the expression
    \begin{equation*}
    \begin{aligned}
        \tau & \left(  \cL_k(\bA) \left( \bf_{1}^{(j)} \otimes \cdots \otimes \bf_{k}^{(j)} \right) \right) \\
        &= \tau \bigg( \left( \bA \bf_{1}^{(j)} \right) \otimes \bf_{2}^{(j)} \otimes \cdots \otimes \bf_{k}^{(j)} + \ldots \\
        &\qquad \qquad + \bf_{1}^{(j)} \otimes \cdots \otimes \bf_{k-1}^{(j)} \otimes \left( \bA \bf_{k}^{(j)} \right) \bigg).
    \end{aligned}
    \end{equation*}
    Applying the permutation $\tau$ to each term yields
    \begin{equation*}
    \begin{aligned}
        \left( \bA \bf_{\tau(1)}^{(j)} \right)& \otimes \bf_{\tau(2)}^{(j)} \otimes \cdots \otimes \bf_{\tau(k)}^{(j)} + \ldots \\
        & \quad + \bf_{\tau(1)}^{(j)} \otimes \cdots \otimes \bf_{\tau(k-1)}^{(j)} \otimes \left( \bA \bf_{\tau(k)}^{(j)} \right) \\
        &= \cL_k(\bA) \, \tau\left(\bf_{1}^{(j)} \otimes \cdots \otimes \bf_{k}^{(j)} \right).
    \end{aligned}
    \end{equation*}
    In summary, we have
    \begin{equation}
        \label{eq:rank1perm}
        \begin{aligned}
            \tau & \left(  \cL_k(\bA) \left( \bf_{1}^{(j)} \otimes \cdots \otimes \bf_{k}^{(j)} \right) \right) \\
            &\qquad \quad = \cL_k(\bA) \, \tau\left(\bf_{1}^{(j)} \otimes \cdots \otimes \bf_{k}^{(j)} \right).
        \end{aligned}
    \end{equation}
    Next, we apply the permutation $\tau$ to both sides of the linear system $\cL_k(\bA)\bw = \bb$:
    \begin{equation*}
         \tau\left(\bb \right) = \tau \left( \cL_k(\bA) \sum_{j = 1}^R \bf_{1}^{(j)} \otimes \cdots \otimes \bf_{k}^{(j)} \right).
    \end{equation*}
    Applying the equation derived in \eqref{eq:rank1perm} we obtain
    \begin{equation*}
        \begin{aligned}
            & \cL_k(\bA) \sum_{j = 1}^R  \tau \left( \bf_{1}^{(j)} \otimes \cdots \otimes \bf_{k}^{(j)} \right) = \cL_k(\bA) \tau \left( \bw \right).
        \end{aligned}
    \end{equation*}
    Finally, it is easily verified that
    \begin{align*}
        \bw^\top (\bx \otimes\cdots \otimes \bx) &= \sum_{j = 1}^R \left( \bx^\top \bf_{1}^{(j)} \right) \cdots \left(\bx^\top \bf_{k}^{(j)} \right) \\
        &= \sum_{j = 1}^R \left( \bx^\top \bf_{\tau(1)}^{(j)} \right) \cdots \left(\bx^\top \bf_{\tau(k)}^{(j)} \right) \\
        &= \tau(\bw)^\top (\bx \otimes\cdots \otimes \bx).
    \end{align*}
\end{proof}
In the context of observability energy functions, Lemma~\ref{proposition:permutedsolutions} states that polynomial coefficients for $\coE$ can be obtained by solving a linear system with the right-hand side $\bC_k$ in \eqref{eq:rhs} or a right-hand side $\tau(\bC_k)$ for some $\tau \in \Pi_k$. In addition, we may even solve for the unique symmetric representation by considering $\Sym[\bC_k]$ as the right-hand side of the linear system. However, as pointed out in Section~\ref{sec:symmetrictensors}, symmetrization increases the size of the CP decomposition by $k!$ terms. Thus, we would like to compute polynomial coefficients that have a low-rank representation and only symmetrize (or treat the symmetrization implictly) in proceeding computations. We note that symmetrization is cheap computationally as it only requires rearranging data. A natural question is therefore: For which permutation can we expect the polynomial coefficients to have the lowest Kronecker rank? The answer to this question is summarized in the following lemma.
\begin{lemma}
    \label{lemma:minimalC}
    Consider the polynomial coefficients $\bc_1,\ldots,\bc_d$ of the LPO system in \eqref{eq:posys}. Assume there exist vectors $\bc_2,\ldots,\bc_d$ such that all of their Kronecker ranks are less than or equal to $R$. Let $k \in \mathbb{N}$. If $k = 2\kappa + 1$ is odd, define
    \begin{equation} \label{eq:Ckodd}
        \widetilde{\bC}_k :=
            2 \sum_{i=1}^{\kappa} \bc_i \otimes \bc_{k-i}.
    \end{equation}
    Otherwise, if $k = 2\kappa$ is even, define
    \begin{equation}  \label{eq:Ckeven}
        \widetilde{\bC}_k :=
            \bc_{\kappa} \otimes \bc_{\kappa} + 2 \sum_{i=1}^{\kappa-1} \bc_i \otimes \bc_{k-i}.
    \end{equation}
    The Kronecker rank of $\widetilde{\bC}_k$ is bounded by $(\kappa-1)R^2 + R$ for $k>2$ and equal to $1$ for $k=2$. In general, there is no permutation $\tau \in \Pi_k$ such that $\tau(\tbC_k)$ has a lower bound for the Kronecker rank than the one for $\tbC_k$.
\end{lemma}
\textcolor{black}{
\begin{proof}
    Let the rank-$R$ CP decomposition for $\bc_i$ and $\bc_{k-i}$ and $i=2,\ldots,\kappa$ be given by
    \begin{align*}
        \bc_i &= \sum_{j=1}^R \bc_{i,1}^{(j)} \otimes \cdots \otimes \bc_{i,i}^{(j)}, \\
        \bc_{k-i} &= \sum_{j=1}^R \bc_{k-i,1}^{(j)} \otimes \cdots \otimes \bc_{k-i,k-i}^{(j)}.
    \end{align*}
    We then write
    \begin{equation*}
        \bc_i \otimes \bc_{k-i} = \sum_{j=1}^R \sum_{m=1}^R \bc_{i,1}^{(j)} \otimes \cdots \otimes \bc_{i,i}^{(j)} \otimes  \bc_{k-i,1}^{(m)} \otimes \cdots \otimes \bc_{k-i,k-i}^{(m)},
    \end{equation*}
    which yields $\krank(\bc_i \otimes \bc_{k-i}) \leq R^2$. For the special case of $i=1$, we have 
    \begin{equation*}
        \bc_1 \otimes \bc_{k-1} = \sum_{j=1}^R \bc_1 \otimes \bc_{k-1,1}^{(j)} \otimes \cdots \otimes \bc_{k-1,k-1}^{(j)}
    \end{equation*}
    and therefore $\krank(\bc_1 \otimes \bc_{k-1}) \leq R$. Combining both of the prior observations yields the proposed bound $\krank(\tbC_k) \leq (\kappa - 1)R^2 + R$ for odd and even values of $k$. The fact that there is no permutation $\tau(\tbC_k)$ which yields a lower bound follows from the fact that the representation for $\tbC_k$ is minimal in the sense that each Kronecker product of the form $\bc_i \otimes \bc_{k-i}$ or $\bc_{k-i} \otimes \bc_i$ only appears once in the sum. Since any permutation $\tau$ can only change the order of the Kronecker products, we cannot reduce the introduced representation to less terms and therefore cannot obtain a generic lower bound  smaller than the proposed one.
\end{proof}}
With this result in hand we summarize our proposed computational procedure for low-rank approximation of the observability energy function coefficients in Algorithm~\ref{alg:quadrature}.

\begin{algorithm}[H]
	\caption{Low-Rank Approximation of Observability Energy Function}
    \label{alg:quadrature}
	\begin{algorithmic}[1]
		\Require $\bA \in \mathbb{R}^{n \times n}$, CP decompositions for $\bc_1, \ldots, \bc_d$, tolerance $\varepsilon > 0$
		\Ensure CP decompositions for $\bw_{2},\ldots,\bw_{2d}$
		\For{$k = 2,\ldots,2d$}
			\State Assemble right-hand side $\widetilde{\bC}_k$ as in~\eqref{eq:Ckodd} or 
   ~\eqref{eq:Ckeven}
			\State Determine $\ell$ based on the priori error bound \eqref{eq:quadratureerror} for the quadrature approximation
			\For{$i = -\ell,\ldots,\ell$}
				\State $\bw_k^{(i)} = \omega_i \big(  \exp(\alpha_i \bA^\top) \otimes \cdots \otimes  \exp(\alpha_i \bA^\top) \big) \widetilde{\bC}_k$
			\EndFor
			\State Add quadrature terms $\bw_{k} = \bw_k^{(-\ell)} + \cdots + \bw_k^{(\ell)}$
		\EndFor
	\end{algorithmic}
\end{algorithm}
Algorithm~\ref{alg:quadrature} approximates the coefficients of the observability energy function
\begin{equation*}
    \coE(\bx) = \frac{1}{2} \sum_{j=2}^{2d} \bw_j^\top (\bx \otimes \cdots \otimes \bx)
\end{equation*}
via vectors in a storage-efficient Kronecker product representation. As pointed out earlier, it is crucial to use such low-rank approximations in a large-scale setting. In particular, we use the proposed computational procedure in Section~\ref{sec:convdiff} and emphasize that storing the energy function coefficients would not be feasible without relying on low-rank decompositions in that case.

\section{Energy-Based MOR}
\label{sec:energybasedMOR}
In this section we discuss how the energy functions in \eqref{eq:controllabilityenergy} and \eqref{eq:observabilityenergy} can be used in the context of MOR and introduce an optimization based perspective on BT. 
\subsection{Balanced Truncation}
\label{sec:linearbt}
We begin by reviewing the BT framework for linear dynamical systems \cite{Moore81}. Consider the LTI system
\begin{equation}
    \begin{aligned}
        \label{eq:LTI}
        \dot{\bx}(t) &=\bA \bx(t) + \bB \bu(t) \\
        y(t) &= \bc^\top \bx(t).
    \end{aligned}
\end{equation}
The controllability and observability energy functions of~\eqref{eq:LTI} are quadratic polynomials given by 
\begin{equation} \label{eq:EcEoLTI}    \ccE(\bx) = \frac{1}{2} \bx^\top \cP^{-1} \bx \quad \text{and} \quad
    \coE(\bx) = \frac{1}{2} \bx^\top \cW_2 \bx.
\end{equation}
In~\eqref{eq:EcEoLTI},
 the controllability gramian $\cP$ solves the Lyapunov equation~\eqref{eq:controllyap} and the observability gramian $\cW_2$ solves the dual Lyapunov equation
\begin{equation*}
    \bA^\top \cW_2 + \cW_2 \bA + \bc\bc^\top = 0.
\end{equation*}
Equivalently, solving the linear system in \eqref{eq:kroneckersystem} yields the observability Gramian via $\vec(\cW_2) = \bw_2$. BT is based on computing a balancing transformation $\bT \in \mathbb{R}^{n \times n}$ that simultaneously diagonalizes the matrices $\cP$ and $\cW_2$. More precisely, the balancing transformation satisfies 
\begin{equation}
    \label{eq:balancingtransformation}
    \bT \cP \bT^{\top} = \bT^{-\top} \cW_2 \bT^{-1} = \Sigma = \operatorname{diag}(\sigma_1,\ldots,\sigma_n).
\end{equation}
The diagonal entries $\sigma_1,\ldots,\sigma_n$ are the Hankel singular values of the system which quantify contributions of each state variable to the input-output behaviour of the system \cite{Antoulas05}. We use $\bT$ to obtain a balanced representation of the system in \eqref{eq:LTI} given by
\begin{equation}
\label{eq:balanced}
\begin{aligned}
    \dot{\tbx}(t) &= \bT \bA \bT^{-1} \tbx(t) + \bT \bB \bu(t) \\
    y(t) &= \bc^\top \bT^{-1} \tbx(t).
\end{aligned}
\end{equation}
Note that the state-space transformation preserves the input-output behavior of the system. However, the energy functions of the transformed (balanced) system  become
\begin{equation*}
    \tccE(\tbx) = \frac{1}{2} \tbx^\top \Sigma^{-1} \tbx \quad \text{and} \quad
    \tcoE(\tbx) = \frac{1}{2} \tbx^\top \Sigma \tbx.
\end{equation*}
This reveals that in the balanced representation of the system the state components are in fact sorted based on their contribution to the energy functions. The idea of BT is to remove parts of the state vector $\tbx$ which correspond to state components that are difficult to control and simultaneously contribute little to the ouput of the system. Typically, MOR of LTI systems is performed by choosing MOR matrices $\bV \in \mathbb{R}^{n \times r}$ and $\bW \in \mathbb{R}^{n \times r}$ with $r \ll n$ and then construcing the ROM by
\begin{equation}
    \begin{aligned}
    \label{eq:ROM}
    \dot{\hbx}(t) &=\hbA \hbx(t) + \hbB \bu(t) \\
    \hy(t) &= \hbc^\top \hbx(t)
\end{aligned}
\end{equation}
with a low-dimensional state $\hbx(t) \in \mathbb{R}^r$ via
\begin{equation} \label{eq:LinProj}
    \hbA = \bW^\top \bA \bV, \quad \hbB = \bW^\top \bB \quad \text{and} \quad \hbc = \bV^\top \bc.
\end{equation}
In~\eqref{eq:LinProj}, without of loss of generality, we assumed that  $\bW^\top \bV = \bI_r$.
In the BT framework with the balanced system~\eqref{eq:balanced}, this corresponds to choosing specific matrices
\begin{equation*}
    \bV = \bT^{-1}(:,1:r) \quad \text{and} \quad \bW^\top = \bT(1:r,:),
\end{equation*}
{where $\bT(1:r,:)$ denotes the matrix formed by the leading $r$ rows of $\bT$ and $\bT^{-1}(:,1:r)$ denotes the matrix formed by the leading $r$ columns of $\bT^{-1}$.} Hence, we truncate $n-r$ state variables of the balanced system representation.
The key properties which make this method appealing are efficient algorithms for computing $\bV$ and $\bW$ as well as an a priori error bound for the input-output error in terms of the Hankel singular values \cite{Antoulas05}. We emphasize that in practice BT is performed \emph{without} forming the fully balanced system~\eqref{eq:balanced} as it is numerically ill-conditioned~\cite{Antoulas05}. Here we present the balance-then-truncate framework to illustrate the theoretical formulation. 
\subsection{An Optimization Perspective on Balanced Truncation}
\label{sec:optimizationperspective}
While simultaneous diagonalization of two matrices is a plausible perspective on BT in the linear case, it does not readily extend to the general nonlinear setting. Instead, diagonal forms of the energy functions are considered in \cite{Scherpen93,Fujimoto10,Kramer23}. These diagonal forms are based on a nonlinear state-space transformation $\Phi$ which satisfies 
\begin{align*}
    \ccE(\Phi(x)) &= \frac{1}{2} \sum_{i=1}^n \frac{x_i^2}{\sigma_i(x_i)}, \\
    \coE(\Phi(x)) &= \frac{1}{2} \sum_{i=1}^n \sigma_i(x_i) x_i^2.
\end{align*}
The state-dependent singular value function $\sigma_i:\R \rightarrow \R$ is used to quantify the importance of the state component $x_i$ for the input-output behavior. This allows for performing MOR by truncating $x_i$ that correspond to (in some appropriate sense) small $\sigma_i$, similarly to the linear BT framework. Unlike in the linear case, however, computing the transformation $\Phi$ accurately is difficult in practice. Additionally, performing a nonlinear state-space transformation does not preserve an underlying nonlinear structure in the model. In general, an LPO system would be transformed into a system with a nonlinear state equation which may make it difficult to interpret and analyze the ROM.
{
A different balancing-based approach for LTI \emph{quadratic output} systems has been introduced in \cite{Benner22}. There the balancing transformation is obtained by simultaneous diagonalization of $\cP$ and a matrix $\cQ$ which satisfies
\begin{equation} \label{eq:BennerQ}
    \bA^\top \cQ + \cQ \bA + \bC_2^\top \cP \bC_2 = 0, 
\end{equation}
where  $\bC_2\in \R^{n \times n}$  defines the quadratic output of the underlying system given by $y(t) = \bx^\top \bC_2 \bx$. The matrix $\cQ$ is related to the observability energy function via
\begin{equation*}
    \coE(\bx) \leq (\bx^\top \cQ \bx) (\bx^\top \cP^{-1} \bx).
\end{equation*}
The rest of the algorithm follows as in regular BT for LTI systems.
While this approach has favorable properties such as an $\cH_2$ error bound, it does not directly take into account the true energy function $\coE$ like the linear and nonlinear BT frameworks do. Therefore, an MOR approach which combines structure preservation (i.e., the ROM is still an  LPO system ) and reduction based on the true energy functions of the underlying system is  still missing.} Here, our goal is to introduce a MOR approach which aims to close this gap. \\

Towards this goal, Lemma~\ref{lemma:optimizationbt} gives an equivalent formulation of BT for linear systems which readily extends to LPO systems. The key idea of this lemma is to view BT as a two-step procedure: First, we consider the Cholseky factorization of $\cP = \bR \bR^\top$ and perform a state
space transformation on~\eqref{eq:LTI} via $\bR$ resulting in the system
\begin{align*}
    \dot{\tbx}(t) &= \bR^{-1} \bA \bR \tbx(t) + \bR^{-1} \bB \bu(t) \\
    y(t) &= \bc^\top \bR \tbx(t).
\end{align*}
The corresponding energy functions of this system are 
\begin{equation}
    \label{eq:linearenergyfunctions}
    \tccE(\tbx) = \frac{1}{2} \tbx^\top \tbx \quad \text{and} \quad
    \tcoE(\tbx) = \frac{1}{2} \tbx^\top \tcW_2 \tbx,
\end{equation}
where $\tcW_2 = \bR^\top \cW_2 \bR$. 
In this representation all components of the state $\tbx$ are equally difficult to control. Hence, the system representation is called an input-normal form. In a second step we seek to find an $r$-dimensional subspace which contains states that maximally contribute to the observability energy in a specific domain. In particular, we look at 
\emph{the average observability energy} in an $n$-dimensional ball of radius $L$ which we denote as $\cB_n(L) := \{ \bx \in \mathbb{R}^n : \lVert \bx \rVert_2 \leq L \}$. In order to make this idea more clear let us assume an orthogonal matrix $\tbQ \in \mathbb{R}^{n \times r}$ is given. The average observability energy in $\cB_n(L)$ of states in $\operatorname{Range}(\tbQ)$ is then given by 
\begin{align*}
    F(\tbQ) :&= \frac{1}{V_n(L)} \int_{\bx \in \cB_n(L)} \tcoE(\tbQ \tbQ^\top \bx) \operatorname{d\bx}\\
    &= \frac{1}{V_n(L)} \int_{\bx \in \cB_n(L)} \bx^\top \tbQ \tbQ^\top \tcW_2 \tbQ \tbQ^\top \bx \operatorname{d\bx},
\end{align*}
where $V_n(L)$ is the volume of $\cB_n(L)$. As we will show in Lemma~\ref{lemma:energytraceformula},  this function can be written as
\begin{equation}
    \label{eq:linearFtrace}
    F(\tbQ) = \frac{L^2}{n+2} \operatorname{trace}(\tbQ^\top \tcW_2 \tbQ).
\end{equation}
Now, a basis for an $r$-dimensional subspace which maximizes the average observability energy in $\cB_n(L)$ can be obtained by solving the optimization problem
\begin{equation}
    \label{eq:linearoptimization}
    \max_{\substack{\tbQ \in \mathbb{R}^{n \times r} \\ \tbQ^\top\tbQ = I}} F(\tbQ).
\end{equation}
In the next lemma, we show that this formulation is indeed equivalent to BT in the case of LTI sytem.
\begin{lemma}
    \label{lemma:optimizationbt}
    Given the LTI system~\eqref{eq:LTI},
    compute the Cholesky factorization of the controllability gramian $\cP = \bR \bR^\top$ and 
    define $\tcW_2 = \bR^\top\cW_2 \bR$. Let
    \begin{equation*}
        \bQ = \argmax_{\substack{\tbQ \in \mathbb{R}^{n \times r} \\ \tbQ^\top\tbQ = I}} F(\tbQ).
    \end{equation*}
    Then the choice of MOR matrices 
    \begin{equation*}
    \bV = \bR \bQ \quad \text{and} \quad \bW^\top = \bQ^\top \bR^{-1}
    \end{equation*}
    in~\eqref{eq:ROM} and~\eqref{eq:LinProj} leads to a ROM which defines the same input-output mapping as the ROM obtained via BT, i.e., it leads to the same ROM obtained via BT.
\end{lemma}
\begin{proof}
    First, recall that the balancing transformation in \eqref{eq:balancingtransformation} can be written as 
    \begin{equation*}
        \bT = \Sigma^{\frac{1}{2}} \bU^\top \bR^{-1}
    \end{equation*}
    where $\bU$ corresponds to the singular vectors of the symmetric positive definite observability energy coefficient matrix $\tcW_2$ of the input-normal energy function in \eqref{eq:linearenergyfunctions}.
    Next, we note that the leading $r$ singular vectors of the matrix $\tcW_2$, i.e., $\bQ = \bU(:,1:r)$,
solves
    the optimization problem proposed in \eqref{eq:linearoptimization}.   This can be easily verified via the trace formula for $F$ specified in \eqref{eq:linearFtrace} \cite{Kokiopoulou2011}. In particular, $\bQ$ is the optimum for any choice of $L$. This means that
    \begin{align*}
        \bV &= \bR \bQ = \bT^{-1}(:,1:r)\Sigma^{-\frac{1}{2}}(1:r,1:r)~~\text{and} \\
        \bW^\top &= \bQ^\top \bR^{-1} = \Sigma^{\frac{1}{2}}(1:r,1:r) \bT(1:r,:).
    \end{align*}
    The projection matrices $\bV$ and $\bW$ are therefore the same as in the BT framework, except that they are multiplied with the diagonal scaling matrices $\Sigma^{-\frac{1}{2}}(1:r,1:r)$ and $\Sigma^{\frac{1}{2}}(1:r,1:r)$, respectively. Hence, the ROMs are equivalent up to a state-space transformation via $\Sigma^{-\frac{1}{2}}(1:r,1:r)$, which means that they define the same input-output mapping.
\end{proof}

\section{Energy-Based Reduction of LPO Systems}
\label{sec:energybasedapproximation}
In this section, we show that the novel energy-based formulation for BT outlined in Lemma \ref{lemma:optimizationbt} readily extends to LPO systems. As in the linear case, we first bring the system in \eqref{eq:posys} into an input-normal form
\begin{equation}
\label{eq:inputnormalltipo}
\begin{aligned}
    \dot{\tbx}(t) &=\bR^{-1} \bA \bR \tbx(t) + \bR^{-1} \bB u(t) \\
    y(t) &= \bc_1^\top \bR \tbx(t) + \cdots + \bc_d^\top \left(\bR \tbx(t) \otimes \cdots \otimes \bR \tbx(t) \right).
\end{aligned}
\end{equation}
This form can easily be computed by determining the Cholesky factorization $\cP = \bR \bR^\top$ and performing tensor-matrix multiplications with the polynomial output coefficients $\bc_k$ and the Cholesky factor $\bR$. The corresponding energy functions in this transformed basis are given by
\begin{align}
    \tccE(\tbx) &= \frac{1}{2} \tbx^\top \tbx, \notag \\
    \tcoE(\tbx) &= \frac{1}{2} \sum_{j=2}^{2d} \tbw_j^\top (\tbx \otimes \cdots \otimes \tbx), \label{eq:inputnormalobsenergy}
\end{align}
where
\begin{equation*}
    \tbw_j = (\bR^\top \otimes \cdots \otimes \bR^\top)\bw_j \quad \text{for} \quad j=1,\ldots,2d.
\end{equation*}
Based on Proposition~\ref{proposition:permutedsolutions}, without loss of generality we assume that $\bw_j$ and therefore $\tbw_j$ are symmetric. Next, we would like to find a matrix $\bQ$ such that $\operatorname{Range}(\bQ)$ represents a subspace of states which maximize the average observability energy in $\cB_n(L)$. Based on such a $\bQ$ we can define $\bV = \bR \bQ$, $\bW^\top = \bQ^\top \bR^{-1}$ and the ROM
\begin{equation}
\label{eq:ltiporom}
\begin{aligned}
    \dot{\hbx}(t) &=\hbA \hbx(t) + \hbB u(t) \\
    \hy(t) &= \hbc_1^\top \hbx(t) + \cdots + \hbc_d^\top \left(\hbx(t) \otimes \cdots \otimes \hbx(t) \right),
\end{aligned}
\end{equation}
where $\hbA = \bW^\top \bA \bV$, $\hbB = \bW^\top \bB$ and $\hbc_k = (\bV^\top \otimes \cdots \otimes \bV^\top) \bc_k$. 

In order to compute the requried $\bQ$ matrix, we consider the optimization problem
\begin{equation}
    \label{eq:optimizationproblem}
     \max_{\substack{\tbQ \in \mathbb{R}^{n \times r} \\ \tbQ^\top\tbQ = I}} F(\tbQ)
\end{equation}
where
\begin{equation}
    \label{eq:objectivefunction}
    F(\tbQ) := \frac{1}{V_n(L)} \int_{\bx \in \cB_n(L)} \tcoE(\tbQ \tbQ^\top \bx) \operatorname{d\bx}.
\end{equation}
The following result gives an explicit representation of $F$ in terms of traces of square matrices associated with the polynomial coefficients $\tbw_k$.
\begin{lemma}
    \label{lemma:energytraceformula}
    Consider the observability energy function $\tcoE$ in \eqref{eq:inputnormalobsenergy} of the input-normal system in \eqref{eq:inputnormalltipo}. Then 
the average observability energy of states in $\operatorname{Range}(\tbQ)$ in the ball $\cB_n(L)$ is given by
    \begin{align*}
        F(\tbQ) &:= \frac{1}{V_n(L)} \int_{\bx \in \cB_n(L)} \tcoE(\tbQ \tbQ^\top \bx) \operatorname{d\bx} \\
        &= \sum_{\kappa=1}^d c_{\kappa}(n,L) F_{2\kappa}(\tbQ),
    \end{align*}
    where 
    \begin{equation} \label{eq:F2k}
        F_{2\kappa}(\tbQ) = \operatorname{tr}((\tbQ \otimes \cdots \otimes \tbQ)^\top\tbW_{2\kappa} (\tbQ \otimes \cdots \otimes \tbQ))
    \end{equation}
    and
    \begin{equation} \label{eq:ck}
        c_\kappa(n,L) :=  L^{2\kappa} \frac{(2\kappa-1)!! n!!}{(n+2\kappa)!!}.
    \end{equation}
    In~\eqref{eq:F2k}, $\tbW_{2\kappa} \in \mathbb{R}^{n^\kappa \times n^\kappa}$ is the canonical square matricization of the symmetric polynomial coefficient tensor $\tbw_{2\kappa}$ 
    and in~\eqref{eq:ck}, $n!! = 1 \times 3 \times \cdots \times n$ if $n$ is odd and $n!! = 2 \times 4 \times \cdots \times n$ if $n$ is even.
\end{lemma}
The proof of Lemma~\ref{lemma:energytraceformula} can be found in \ref{sec:appendixA}. Note that the symmetric property of $\tbw_{2\kappa}$ is crucial to arrive at the stated trace formula for $F_{2\kappa}$. Additionally, if $\tbw_{2\kappa}$ is given in terms of a CP decomposition, the trace formula can be efficiently evaluated. Another noteworthy fact is that the odd polynomial terms $\tbw_3,\tbw_5,\ldots$ do not appear in our representation of $F(\tbQ)$. 

While the optimization problem in \eqref{eq:linearoptimization} has a closed-form solution for the LTI systems as in~\eqref{eq:LTI}, the closed-form solutions of \eqref{eq:optimizationproblem} in terms of specific Kronecker structured decompositions of $\tbw_k$ are not available for LPO systems. Instead, we propose using a gradient-based optimization technique in order to compute the optimal $\bQ$ of the optimization problem in \eqref{eq:optimizationproblem}. First, we point out that the feasible domain for our optimization problem is a compact Riemannian submanifold of $\mathbb{R}^{n \times r}$ called the Stiefel manifold
\begin{equation*}
    \operatorname{St}(n,r) = \{\bQ \in \mathbb{R}^{n\times r} : \bQ^\top \bQ = \bI\}.
\end{equation*}
Optimization on Riemannian manifolds has been studied rigorously \cite{Boumal2023} and various algorithms have been adapted to work on various Riemannian manifolds. For our optimization problem we use an adaption of the trust region method implemented in the Manopt Matlab toolbox \cite{Boumal2014}. An important prerequisite for using the optimizer is to have access to the gradient of the objective function $F(\tbQ)$. A formula for the gradient $\nabla F$ is given in Lemma~\ref{lemma:gradient}.
\begin{lemma}
    \label{lemma:gradient}
    Consider the function $F$ in \eqref{eq:objectivefunction}. Its gradient is given by
    \begin{equation*}
        \nabla F(\tbQ) = \sum_{\kappa=1}^d c_{\kappa}(n,L) \nabla F_{2 \kappa}(\tbQ),
    \end{equation*}
    where
    \begin{equation*}
        \nabla F_{2 \kappa}(\tbQ) = \sum_{\ell=1}^r \operatorname{mat}(\tbw_{2\kappa})(\tbQ \otimes \tbQ(:,\ell) \otimes \cdots \otimes \tbQ(:,\ell))
    \end{equation*}
    and $\operatorname{mat}(\tbw_{2\kappa}) \in \mathbb{R}^{n \times n^{2\kappa - 1}}$ refers to the mode-$1$ matricization of the underlying tensor \cite{Kolda09}.
\end{lemma}
We refer to \ref{sec:appendixB} for a proof of Lemma~\ref{lemma:gradient}. The gradient can be evaluated efficiently if $\tbw_{2\kappa}$ is given in terms of a CP decomposition. Further, we point out that we are considering a non-convex optimization problem in \eqref{eq:optimizationproblem}. This means that an optimization algorithm typically only converges to a local minimum and it is beneficial to pick a good initial guess for $\bQ$. Here we propose to pick the leading $r$ singular vectors of $\tcW_2$ as an initial guess for $\bQ$. This corresponds to the $\bQ$ used in the linear BT framework.\\

Now we have all the components of our energy-based  MOR procedure for LPO systems. We outline the resulting method in Algorithm~\ref{alg:energyMOR}.
\begin{algorithm}[H]
	\caption{Energy-Based Reduction of LPO System}
    \label{alg:energyMOR}
	\begin{algorithmic}[1]
		\Require LPO system as in \eqref{eq:posys}, $r\in\mathbb{N}$, $L > 0$
		\Ensure Reduced order LPO system as in \eqref{eq:ltiporom}
        \State Solve
        \begin{equation*}
            	0 = \bA \cP + \cP \bA^\top + \bB \bB^\top
        \end{equation*}
        for $\cP = \bR \bR^\top$
        \State Compute $\bw_2,\bw_4,\ldots$ and symmetrize
        \State Solve 
        \begin{equation*}
            \bQ = \argmax_{\substack{\tbQ \in \mathbb{R}^{n \times r} \\ \tbQ^\top \tbQ = I}} F(\tbQ)
        \end{equation*}
        \State Form $\bV = \bR \bQ$ and $\bW^\top = \bQ^\top \bR^{-1}$
        \State Compute the ROM as in \eqref{eq:ltiporom}
	\end{algorithmic}
\end{algorithm}
{While both nonlinear BT~\cite{Scherpen93} and the proposed method in Algorithm~\ref{alg:energyMOR} utilize the energy functions of an underlying dynamical system to compute a ROM, there are fundamental differences between the two approaches. Nonliear BT diagonalizes the energy functions in order to determine state components that are significant for the input-output mapping.  Algorithm~\ref{alg:energyMOR}, instead, avoids the diagonalization step and directly yields a ROM in which state components have a high contribution to the observability energy. In particular, our proposed approach does not yield singular value functions along the lines of nonlinear BT. It should be pointed out that the computation of these singular value functions is difficult in the large-scale setting and strategies for ordering them for the purpose of truncation-based MOR are still not fully resolved. 
Moreover, while nonlinear BT will lead to a fully nonlinear ROM, the proposed ROM will retain the LPO structure. Another aspect worth pointing out is that the nonlinear BT framework guarantees that the ROM is locally asymptotically stable. While all of our numerical experiments using Algorithm~\ref{alg:energyMOR} yield asymptotically stable ROMs, it remains an open question whether asymptotic stability of ROMs is guaranteed with our proposed method.}
\section{Numerical Examples}
\label{sec:numericalexamples}
In this section we apply our proposed MOR procedure to two benchmark problems.
\subsection{Mass-Spring-Damper System}
We begin by considering an LPO system 
\begin{equation*}
\begin{aligned}
    \dot{\bx}(t) &=\bA \bx(t) + \bB \bu(t), \\
    y(t) &= \bc_1^\top \bx(t) + \bc_2^\top \left(\bx(t) \otimes \bx(t) \right)
\end{aligned}
\end{equation*}
with linear and quadratic terms in the output. The model is a modified version of the mass-spring-damper system considered in \cite{Gugercin2012}. In particular, we consider here a modified output which combines the Hamiltonian of the system given by $\cH(\bx(t)) = \bc_2^\top \left(\bx(t) \otimes \bx(t) \right)$ and the first linear output $y_1(t) = \bc_1^\top\bx(t)$ as specified in \cite{Gugercin2012}. Further, we use $n=50$ as the dimension of our full order model and compute a ROM of dimension $r=10$ using Algorithm~\ref{alg:energyMOR} for various choices of $L$. In order to compare our method to an existing MOR approach we also compute a ROM of order $r=10$ using the BT variant for quadratic output systems introduced in \cite{Benner22}, which we call QOBT here. 
(Recall that QOBT uses  $\cP$ in~\eqref{eq:controllyap} and the modified observability Gramian 
$\cQ$  in~\eqref{eq:BennerQ}).
The system has $2$ inputs. To test the accuracy of ROMs, we choose  $u_1(t) = u_2(t) = \exp(-2t)\sin(\frac12 t)$ and simulate, for $t \in [0,20]$, the full order model (FOM) and the ROMs obtained by Algorithm~\ref{alg:energyMOR} as well as QOBT.
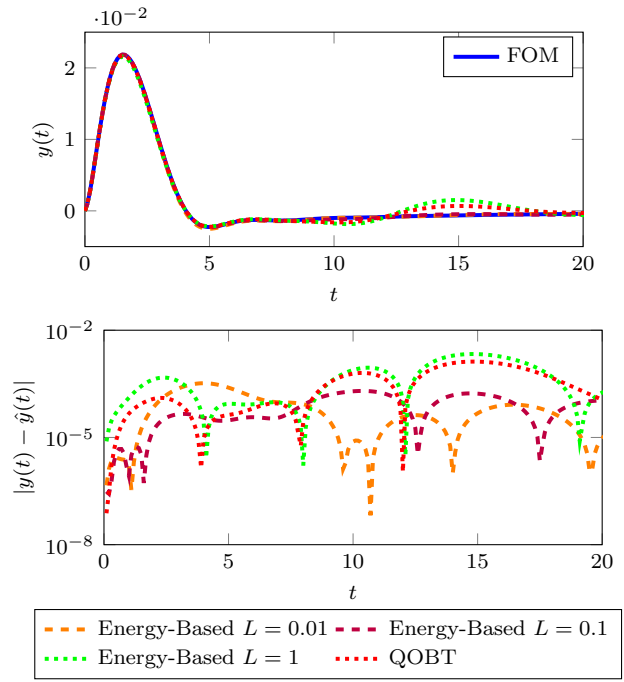
\begin{figure}
    \centering
\begin{tikzpicture}[font = \plotfontsize]
      \pgfplotstableread{figures/msd.dat}\tableINPUT
      \begin{axis}[%
        width  = .75\linewidth,
        height = .12\textheight,
        scale only axis,
        xmin = 0,
        xmax = 20,
        ymin = -5e-3,
        ymax = 2.5e-2,
        xminorticks = false,
        yminorticks = false,
        xlabel = {$t$},
        ylabel = {$y(t)$},
        ylabel style   = {yshift = -.3em},
        scaled x ticks = false,
        x tick label style = {/pgf/number format/1000 sep={\,}},
        y tick label style = {/pgf/number format/1000 sep={\,}},
        cycle list name    = plotlistmsd,
        legend style={anchor=north east},
        legend cell align={left}
      ]
    
        \foreach \y in {1, 2, 3, 4, 5}{
          \addplot+ table[x index = 0, y index = \y] {\tableINPUT};
        }
        \legend{FOM}
      \end{axis}
    \end{tikzpicture}
    \begin{tikzpicture}[font = \plotfontsize]
      \pgfplotstableread{figures/msd_error.dat}\tableINPUT
      \begin{axis}[%
        width  = .75\linewidth,
        height = .12\textheight,
        scale only axis,
        ymode = log,
        xmin = 0,
        xmax = 20,
        ymin = 1e-8,
        ymax = 1e-2,
        xminorticks = false,
        yminorticks = false,
        xlabel = {$t$},
        ylabel = {$|y(t) - \hy(t)|$},
        ylabel style   = {yshift = -.3em},
        scaled x ticks = false,
        x tick label style = {/pgf/number format/1000 sep={\,}},
        y tick label style = {/pgf/number format/1000 sep={\,}},
        cycle list name    = plotlistmsderror,
        legend style={legend columns=2,at={(0.45,-0.3)},anchor=north},
        legend cell align={left}
      ]
    
        \foreach \y in {1, 2, 3, 4}{
          \addplot+ table[x index = 0, y index = \y] {\tableINPUT};
        }
        \legend{Energy-Based $L=0.01$,Energy-Based $L=0.1$,Energy-Based $L=1$, QOBT}
      \end{axis}
    \end{tikzpicture}
    \caption{Comparison of time-domain responses for mass-spring-damper system.}
    \label{fig:msd}
\end{figure}
We see in the top plot of Figure~\ref{fig:msd} that all of the approaches capture the output of our benchmark problem well. The bottom plot of Figure~\ref{fig:msd} shows that the performance of the ROMs computed by Algorithm~\ref{alg:energyMOR} greatly depends on the choice $L$. In particular, the choice $L=0.1$ has the smallest $\cL_\infty$ error $\lVert y - \hat{y} \rVert_{\cL_\infty}$ and has a lower error than QOBT for most time steps. The choice $L=0.01$ performs well when the output is close to zero but has a larger output error in the interval $t \in [0,5]$ where $y(t)$ has a peak. Overall, the choice $L=1$ performs slightly worse than QOBT here. All of these observations match the local nature of our proposed approach. {Especially, we may explain the superior performance for the choice of $L=0.1$ by considering that the states for our choice of $\bu$ are contained in the ball $\cB_n(0.102)$. This means we maximize the observability energy in a subdomain of the state-space which contains states that are particularly relevant for the time-domain simulation considered here.} Additionally, all of the computed ROMs were asymptotically stable. While there are theoretical guarantees for this fact in the QOBT framework, it is not yet clear whether this always holds for our proposed MOR procedure.

\subsection{Convection Diffusion Model}
\label{sec:convdiff}
\begin{figure}
    \centering
    \begin{tikzpicture}[font = \plotfontsize]
      \pgfplotstableread{figures/conv_diff_cubic.dat}\tableINPUT
      \begin{axis}[%
        width  = .75\linewidth,
        height = .12\textheight,
        scale only axis,
        xmin = 0,
        xmax = 10,
        ymin = -0.2,
        ymax = 0.6,
        xminorticks = false,
        yminorticks = false,
        xlabel = {$t$},
        ylabel = {$y(t)$},
        ylabel style   = {yshift = -.3em},
        scaled x ticks = false,
        x tick label style = {/pgf/number format/1000 sep={\,}},
        y tick label style = {/pgf/number format/1000 sep={\,}},
        cycle list name    = plotlistconvdiff,
        legend style={anchor=north east},
        legend cell align={left}
      ]
    
        \foreach \y in {1, 2}{
          \addplot+ table[x index = 0, y index = \y] {\tableINPUT};
        }
        \legend{FOM, ROM}
      \end{axis}
    \end{tikzpicture}
    \begin{tikzpicture}[font = \plotfontsize]
      \pgfplotstableread{figures/conv_diff_cubic_error.dat}\tableINPUT
      \begin{axis}[%
        width  = .75\linewidth,
        height = .12\textheight,
        scale only axis,
        ymode = log,
        xmin = 0,
        xmax = 10,
        ymin = 1e-10,
        ymax = 1e-4,
        xminorticks = false,
        yminorticks = false,
        xlabel = {$t$},
        ylabel = {$|y(t) - \hy(t)|$},
        ylabel style   = {yshift = -.3em},
        scaled x ticks = false,
        x tick label style = {/pgf/number format/1000 sep={\,}},
        y tick label style = {/pgf/number format/1000 sep={\,}},
        cycle list name    = plotlistconvdifferror,
        legend style={at={(0.9825,0.2275)},anchor=north east},
        legend cell align={left}
      ]
    
        \foreach \y in {1}{
          \addplot+ table[x index = 0, y index = \y] {\tableINPUT};
        }
        \legend{Error}
      \end{axis}
    \end{tikzpicture}
    \caption{Comparison of time-domain responses for convection diffusion model.}
    \label{fig:convdiff}
\end{figure}
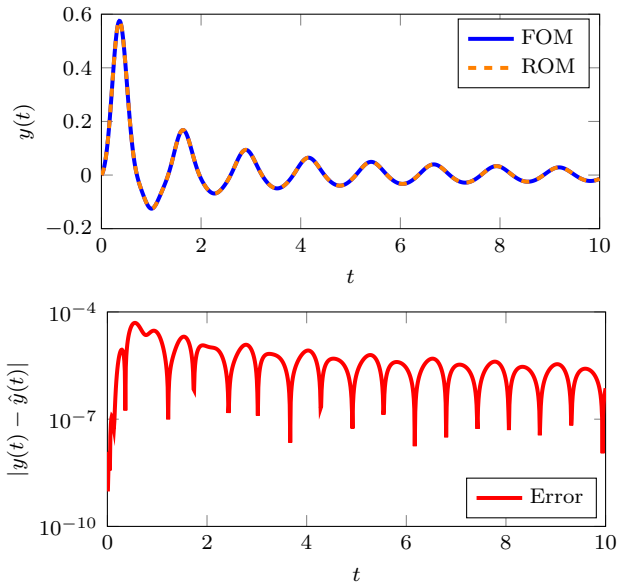
Next, we consider the convection diffusion equation
\begin{equation*}
    \begin{aligned}
        \dot{c} - \Delta c + \bv \cdot \nabla c &= f \quad \text{in } \Omega = [0,1] \times [0,1], \\
        c &= 0 \quad \text{on } \partial \Omega.
    \end{aligned}
\end{equation*}
as a benchmark problem. We use a spatial discretization of the above PDE using an equidistant mesh and finite differences leading to the same system matrices as the ones considered in \cite{Kressner10}. In the following we choose $v=1$ for the convection coefficient and an output corresponding to the degree-$3$ polynomial
\begin{align*}
    y(t) &= \bc_1^\top \bx(t) + \bc_2^\top \left(\bx(t) \otimes \bx(t) \right) \\
    &\, + \bc_3^\top \left(\bx(t) \otimes \bx(t) \otimes \bx(t) \right).
\end{align*}
The output coefficients are defined as $\bc_1 = 10\be_1$, $\bc_2 = 100(\be_2 \otimes \be_2)$ and $\bc_3 = 1000(\be_3 \otimes \be_3 \otimes \be_3)$, where $\be_k \in \R^n$ is the $k$-th unit vector. As in \eqref{eq:posys} the input enters the state equation via $\bB u(t)$ where $\bB = \mathbbm{1}_n$ is the vector of all ones. As a system dimension we use $n=2000$ and we reduce to an order $r=15$ model using Algorithm~\ref{alg:energyMOR} with $L=1$. We emphasize that for this example we need to compute the coefficients of $\coE$ in~\eqref{eq:observabilitypolynomial} up to the vector $\bw_6$ since $d=3$ in this case. For the system dimension $n=2000$ considered here, naively solving the systems in \eqref{eq:kroneckersystem} and storing the full vectors is not feasible. For example, $\bw_6$ alone would require $2.56 \times 10^{11}$ Gigabytes of storage in single precision. Therefore, it is crucial to leverage low-rank solvers according to our discussion in Section~\ref{sec:lrsolutions}. In particular, we leverage Algorithm~\ref{alg:quadrature} for the computations performed in this section. We visualize the time-domain response using the input $u(t) = \frac{100}{t+1}\sin(5t)$ in the top of Figure~\ref{fig:convdiff} and the corresponding error in the bottom of the Figure. The ROM captures the input-output behavior well for the significant reduction in the model order that we consider here. Again, the ROM preserves the asymptotic stability of the FOM.

\section{Conclusion}
\label{sec:conclusion}
We showed that energy functions of LPO systems are again polynomials and gave explicit formulae for computing the polynomial coefficients. The corresponding linear systems have a Kronecker product structure and can in certain situations be solved efficiently using low-rank CP decompositions. We discussed a new perspective on the BT method for linear systems and extended it to LPO systems, yielding an effective new MOR technique for approximating linear systems with polynomial outputs. The effectiveness of the proposed algorithm were illustrated on two numerical examples.

\section*{Acknowledgments}
{This work was supported in part by the
National Science Foundation (NSF), United States under Grant No. CMMI-2130695.}

\bibliographystyle{plainurl}
\bibliography{references}

\appendix
\section{Proof of Lemma~\ref{lemma:energytraceformula}}
\label{sec:appendixA}
We begin by considering the homogeneous degree-$2\kappa$ polynomial
\begin{equation*}
    P_{2\kappa}(\bx) = \tbw_{2\kappa}^\top (\bx \otimes \cdots \otimes \bx),
\end{equation*}
where we assume that $\tbw_{2\kappa}$ is symmetric. The work by Folland \cite{Folland2001} derives equations for homogeneous polynomials over $n$-balls and $n$-spheres, which gives
\begin{equation*}
    \int_{\bx \in \cB_n(L)} P_{2\kappa}(\bx) \operatorname{d\bx} = \frac{L^{2\kappa + n}}{2\kappa + n} \int_{\bx \in \cS_n} P_{2\kappa}(\bx) \operatorname{d\bs(\bx)}.
\end{equation*}
Here, $\operatorname{d\bs(\bx)}$ denotes the surface measure on the $n$-sphere $\cS_n = \{ \bx \in \mathbb{R}^n : \lVert \bx \rVert_2 = 1\}$. Since $\tbw_{2\kappa}$ is symmetric it can be written in terms of a symmetric CP decomposition \cite{Comon08}
\begin{equation}
    \label{eq:symmetriccp}
    \tbw_{2\kappa} = \sum_{j=1}^R f_j \otimes \cdots \otimes f_j.
\end{equation}
In particular this means that we can write
\begin{equation}
    \label{eq:polynomialsum}
    P_{2\kappa}(\bx) = \sum_{j=1}^R(f_j^\top \bx)^{2\kappa}.
\end{equation}
Next, we consider the integral of the expression in \eqref{eq:polynomialsum}. The formula
\begin{equation}
    \label{eq:sphereintegral}
    \int_{\bx \in \cS_n} P_{2\kappa}(\bx) \operatorname{d\bs(\bx)} = \sum_{j=1}^R \lVert f_{j} \rVert^{2\kappa} \frac{(2\kappa-1)!! (n-2)!!}{(n+2\kappa-2)!!} A_{n}
\end{equation}
for this integral is given in \cite[Eq. 6.29]{Waldron2018}, where $A_{n}$ is the area of the hyper-surface $\cS_n$. Next, we use the fact that $\lVert f_{j} \rVert^2 = f_{j}^\top f_{j} = \operatorname{tr}(f_{j}^\top f_{j})$ to write
\begin{equation}
\begin{aligned}
    \label{eq:cptrace}
    \sum_{j=1}^{R} \lVert f_{j} \rVert^{2\kappa} &=  \sum_{j=1}^{R} \operatorname{tr}((f_{j}^\top f_{j})^\kappa) \\
    &= \sum_{j=1}^R \operatorname{tr}\left((f_{j} \otimes \cdots \otimes f_{j})^\top (f_{j} \otimes \cdots \otimes f_{j})\right).
\end{aligned}
\end{equation}
Then, we define
\begin{equation*}
    \tbW_{2\kappa} = \sum_{j=1}^R(f_{j} \otimes \cdots \otimes f_{j})(f_{j} \otimes \cdots \otimes f_{j})^\top \in \mathbb{R}^{n^\kappa \times n^\kappa}.
\end{equation*}
By considering the linear and cyclic property of the trace we note that \eqref{eq:cptrace} is equal to $\operatorname{tr}(\tbW_{2\kappa})$:
\begin{equation}
    \label{eq:cptotrace}
    \sum_{j=1}^{R} \lVert f_{j} \rVert^{2\kappa} = \operatorname{tr}(\tbW_{2\kappa}).
\end{equation}
Note that $\tbW_{2\kappa}$ can be obtained via the \texttt{reshape} command in Matlab as 
$    \tbW_{2\kappa} = \texttt{reshape}(\bw_{2\kappa},n^\kappa,n^\kappa)$.
Next, we insert $\eqref{eq:cptotrace}$ into \eqref{eq:sphereintegral}, divide both sides by $V_n(L)$ and use the fact that $A_n/V_n(L) = n/L^n$ \cite{Folland2001} to obtain
\begin{align*}
    \frac{1}{V_n(L)}&\int_{\bx \in \cB_n(L)} P_{2\kappa}(\bx) \operatorname{d\bx} \\
    &= \frac{L^{2\kappa + n}}{2\kappa + n}  \frac{(2\kappa-1)!! (n-2)!!}{(n+2\kappa-2)!!} \frac{A_{n}}{V_n(L)} \operatorname{tr}(\tbW_{2\kappa}) \\
    &= L^{2\kappa} \frac{(2\kappa-1)!! n!!}{(n+2\kappa)!!} \operatorname{tr}(\tbW_{2\kappa}).
\end{align*}
Finally, we use the fact that integrals of homogeneous polynomials with odd degree over $\cB_n$ are zero \cite{Folland2001} which yields
\begin{align*}
F(\tbQ) &= \frac{1}{V_n(L)}\sum_{\kappa=1}^d \int_{\bx \in \cB_n(L)} P_{2\kappa}(\tbQ \tbQ^\top \bx) \operatorname{d\bx} \\
 &= \sum_{\kappa=1}^d c_{\kappa}(n,L) F_{2\kappa}(\tbQ),
\end{align*}
where 
\begin{equation*}
    F_{2\kappa}(\tbQ) = \operatorname{tr}((\tbQ \otimes \cdots \otimes \tbQ)^\top\tbW_{2\kappa} (\tbQ \otimes \cdots \otimes \tbQ)).
\end{equation*}
\section{Derivation of Gradient Formula}
\label{sec:appendixB}
Here we derive an expression for the gradient of
\begin{align*}
    F_{2\kappa}(\tbQ) = \operatorname{tr}((\tbQ \otimes \cdots \otimes \tbQ)^\top\tbW_{2\kappa} (\tbQ \otimes \cdots \otimes \tbQ)).
\end{align*}
As in \ref{sec:appendixA}, we assume that $\tbW_{2\kappa}$ is the square matricization of $\tbw_{2\kappa}$, which has a symmetric CP decomposition along the lines of \eqref{eq:symmetriccp}. This allows for writing
\begin{equation*}
    F_{2\kappa}(\tbQ) = \sum_{j=1}^R \lVert \tbQ^\top f_{j} \rVert^{2\kappa}.
\end{equation*}
The gradient can then be written as
\begin{align*}
    \nabla F_{2\kappa}(\tbQ) &= 2\kappa \sum_{j=1}^R \lVert \tbQ^\top f_{j} \rVert^{2\kappa - 2} f_{j} f_j^\top \tbQ \\
    &= 2\kappa \sum_{j=1}^R \sum_{i=1}^r (f_j \otimes \cdots \otimes f_j)^\top \\
    & \hspace{2cm} (q_i\otimes \cdots \otimes q_i) f_{j} f_j^\top \tbQ \\
    &= 2\kappa \sum_{i=1}^r \sum_{j=1}^R (f_j^\top \otimes \cdots \otimes f_j^\top \otimes f_j) \\ 
    & \hspace{2cm} (q_i \otimes \cdots \otimes q_i \otimes \tbQ) \\
    &= 2\kappa \sum_{i=1}^r \operatorname{mat}(\tbw_{2\kappa}) (q_i \otimes \cdots \otimes q_i \otimes \tbQ),
\end{align*}
where $q_i := \tbQ(:,i)$ and $\operatorname{mat}(\cdot)$ denotes the mode-$1$ matricization of the input tensor \cite{Kolda09}.

\end{document}